\documentclass[a4paper, 12pt]{article}

\usepackage[sort&compress]{natbib}
\bibpunct{(}{)}{;}{a}{}{,} 

\usepackage{amsthm, amsmath, amssymb, mathrsfs, multirow, url, subfigure}
\usepackage{graphicx} 
\usepackage{ifthen} 
\usepackage{amsfonts}
\usepackage[usenames]{color}
\usepackage{fullpage}

\theoremstyle{plain} 
\newtheorem{thm}{Theorem}

\newtheorem{lem}{Lemma}
\newtheorem*{thm0}{Theorem 5$^\prime$}

\theoremstyle{definition}

\theoremstyle{remark}
\newtheorem{remark}{Remark}
\newtheorem{ex}{Example}

\newcommand{\RR}{\mathbb{R}}

\newcommand{\M}{\mathcal{M}}

\newcommand{\E}{\mathsf{E}}

\newcommand{\prob}{\mathsf{P}}

\newcommand{\eps}{\varepsilon}
\renewcommand{\phi}{\varphi}

\newcommand{\bigmid}{\; \bigl\vert \;}

\newcommand{\nm}{\mathsf{N}}
\newcommand{\bet}{\mathsf{Beta}}

\newcommand{\bin}{\mathsf{Bin}}

\newcommand{\chisq}{\mathsf{ChiSq}}

\title{Empirical Bayes posterior concentration in sparse high-dimensional linear models}
\author{
Ryan Martin\footnote{As of August 2016, RM is affiliated with North Carolina State University, {\tt rgmarti3@ncsu.edu}.} \qquad Raymond Mess \\
Department of Mathematics, Statistics, and Computer Science \\
University of Illinois at Chicago \\
{\tt (rgmartin, rmess1)@uic.edu} \\
\mbox{} \\
Stephen G. Walker \\
Department of Mathematics \\
University of Texas at Austin \\
{\tt s.g.walker@math.utexas.edu}
}
\date{\today}

\begin{document}

\maketitle 


\begin{abstract}
We propose a new empirical Bayes approach for inference in the $p \gg n$ normal linear model.  The novelty is the use of data in the prior in two ways, for centering and regularization.  Under suitable sparsity assumptions, we establish a variety of concentration rate results for the empirical Bayes posterior distribution, relevant for both estimation and model selection.  Computation is straightforward and fast, and simulation results demonstrate the strong finite-sample performance of the empirical Bayes model selection procedure.   

\smallskip

\emph{Keywords and phrases:} Data-dependent prior; fractional likelihood; minimax rate; regression; variable selection.
\end{abstract}

\section{Introduction}
\label{S:intro}

In this paper, we consider the Gaussian linear regression model, given by
\begin{equation}
\label{eq:reg}
Y = X \beta + \eps, 
\end{equation}
where $Y$ is a $n \times 1$ vector of response variables, $X$ is a $n \times p$ matrix of predictor variables, $\beta$ is a $p \times 1$ vector of slope coefficients, and $\eps$ is a $n \times 1$ vector of iid $\nm(0,\sigma^2)$ random errors.  Recently, there has been considerable interest in the high-dimensional case, where $p \gg n$, driven primarily by challenging applications.  Indeed, in genetic studies, where the response variable corresponds to a particular observable trait, the number of subjects, $n$, may be of order $10^3$, while the number of genetic features, $p$, in consideration can be of order $10^5$.  Despite the large number of features, usually only a few have a genuine association with the trait.  For example, the \citet{wellcome2007} has confirmed that only seven genes have a non-negligible association with Type~I diabetes.  Therefore, it is reasonable to assume that $\beta$ is sparse, i.e., only a few non-zero entries.  

Given the practical importance of the high-dimensional regression problem, there is now a substantial body of literature on the subject.  In the frequentist setting, a variety of methods are available based on minimizing loss functions, equipped with a penalty on the complexity of the model.  This includes the lasso \citep{tibshirani1996}, the smoothly clipped absolute deviation \citep{fanli2001}, the adaptive lasso \citep{zou2006}, and the Dantzig selector \citep{candes.tao.2007, james.radchenko.2009, james.radchenko.lv.2009}.  \citet{fan.lv.2010} give a selective overview of these and other frequentist methods.  From a Bayesian perspective, popular methods for variable selection in high-dimensional regression include stochastic search variable selection \citep{george.mccullogh.1993} and the methods based on spike-and-slab priors \citep{ishwaran.rao.2005.aos, ishwaran.rao.2005.jasa}.  These methods and others are reviewed in \citet{clydegeorge2004} and \citet{heaton.scott.2009}.  More recently, \citet{bondell.reich.2012}, \citet{johnson.rossell.2012}, and \citet{narisetty.he.2014} propose Bayesian variable selection methods and establish model selection consistency.  

Any Bayesian approach to the regression problem \eqref{eq:reg} yields a posterior distribution on the high-dimensional parameter $\beta$.  It is natural to ask under what conditions will the $\beta$ posterior distribution concentrate around the true value at an appropriate or optimal rate.  Recently, \citet{castillo.schmidt.vaart.2014} show that, with a suitable Laplace-like prior for $\beta$, similar to those in \citet{park.casella.2008}, and under conditions on the design matrix $X$, the posterior distribution concentrates around the truth at rates that match those for the corresponding lasso estimator \citep[e.g.,][]{buhlmann.geer.book}.  These results leave room for improvement in at least two directions; first, the rates associated with the lasso estimator are not optimal, so a break from the Laplace priors (and perhaps even the standard Bayesian setup itself) is desirable; second, and perhaps most importantly, posterior computation with these inconvenient non-conjugate priors is expensive and non-trivial.  In this paper, we develop a new approach, motivated by computational considerations, which leads to improvements in both directions, simultaneously.  

Towards a model that leads to more efficient computation, it is natural to consider a conjugate normal prior for $\beta$.  However, Theorem~2.8 in \citet{castillo.vaart.2012} says that if the prior has normal tails, then the posterior concentration rates can be suboptimal, motivating a departure from the somewhat rigid Bayesian framework.  Following \citet{martin.walker.eb}, we consider a new empirical Bayes approach, motivated by the very simple idea that the tails of the prior are irrelevant as long as its center is chosen informatively.  So, our proposal is to use the data to provide an informative center for the normal prior for $\beta$, along with an extra regularization step to prevent the posterior from tracking the data too closely.  Details of our proposed empirical Bayes model are presented in Section~\ref{S:model}.  It turns out that this new empirical Bayes posterior is both easy to compute and has desirable asymptotic concentration properties.  Section~\ref{S:theory} presents a variety of concentration rate results for our empirical Bayes posterior.  For example, under almost no conditions on the model or design matrix, a concentration rate relative to prediction error loss is obtained which is, at least in some cases, minimax optimal; the optimal rate can be achieved in all cases, but at a cost (see Remark~\ref{re:minimax}).  Furthermore, we provide a model selection consistency result which says that, under optimal conditions, the empirical Bayes posterior can asymptotically identify those truly non-zero coefficients in the linear model.  Our approach has some similarities with the exponential weighting methods in, e.g., \citet{rigollet.tsybakov.2011, rigollet.tsybakov.2012} and \citet{ariascastro.lounici.2014}; in fact, ours can be viewed as a generalization of these approaches, defining a full posterior that, when suitably summarized, corresponds essentially to their estimators.  In Section~\ref{S:numerical} we propose a simple and efficient Markov chain Monte Carlo method to sample from our empirical Bayes posterior, and we present several simulation studies to highlight both the computational speed the superior finite-sample performance of our method compared to several others in terms of model selection.  Finally Section~\ref{S:discuss} gives a brief a discussion, the key message being that we get provable posterior concentration results, optimal in a minimax sense in some cases, fast and easy computation, and strong finite-sample performance.  Lengthy proofs and some auxiliary results are given in the Appendix.

\section{The empirical Bayes model}
\label{S:model}

\subsection{The prior}
\label{SS:prior}

Here, and in the theoretical analysis in Section~\ref{S:theory}, we take the error variance $\sigma^2$ to be known, as is often done \citep[e.g.,][]{rigollet.tsybakov.2012, castillo.schmidt.vaart.2014}.  Techniques for estimating $\sigma^2$ in the high-dimensional case are available; see Section~\ref{S:numerical}.  To specify a prior for $\beta$ that incorporates sparsity, we decompose $\beta$ as $(S,\beta_S)$, where $S \subset \{1,\ldots,p\}$ denotes the ``active set'' of variables, $S=\{j: \beta_j \neq 0\}$, and $\beta_S$ is the $|S|$-vector containing the particular non-zero values.  Based on this decomposition, we can specify the prior for $\beta$ in two steps: a prior for $S$ and then a prior for $\beta_S$, given $S$.  

First, the prior $\pi(S)$ for the model $S$ decomposes as follows:
\begin{equation}
\label{eq:prior.decomp}
\pi(S) = \textstyle\binom{p}{s}^{-1} \, f_n(s), \quad s=0,1,\ldots,p, \quad s=|S|,
\end{equation}
where $f_n(s)$ is a probability mass function on the size $|S|$ of $S$.  That is, we assign a prior distribution $f_n(s)$ on the model size and then, given the size, put a uniform prior on all models of the given size.  Some conditions on $f_n(s)$ will be required for suitable posterior concentration.  In particular, we assume that $f_n(s)$ is supported on $\{0,1,\ldots,R\}$, not on $\{0,1,\ldots,p\}$, where $R \leq n$ is the the rank of the matrix $X$; see, also, \citet{jiang2007}, \citet{abramovich.grinshtein.2010}, \citet{rigollet.tsybakov.2012}, and \citet{ariascastro.lounici.2014}.  That is, 
\begin{equation}
\label{eq:prior.support}
f_n(s) = 0 \quad \text{for all $s=R+1,\ldots,p$}. 
\end{equation}
Our primary motivation for imposing this constraint is that in practical applications, the true value of $s$; i.e. $s^\star=|S^\star|$, is typically much smaller than $R$.  Even in the ideal case where $S^\star$ is known, if $|S^\star| > R$, then quality estimation of the corresponding parameters is not possible.  Moreover, models containing a large number of variables can be difficult to interpret.  Therefore, since having no more variables than samples in the fixed-model case is a reasonable assumption, we do not believe that restricting the support of our prior for the model size is a strong condition.  

Second, for the conditional prior on $\beta_S$, given $S$ that satisfies $|S| \leq R$, we propose to employ the available distribution theory for the least squares estimator $\widehat\beta_S$.  Specifically, we take the prior for $\beta_S$, given $S$, as 
\begin{equation}
\label{eq:conditional.prior}
\beta_S \mid S \sim \nm_{|S|}\bigl( \widehat\beta_S, \gamma^{-1}(X_S^\top X_S)^{-1} \bigr).
\end{equation}
Here, $X_S$ is the matrix filled with columns of $X$ corresponding to $S$, and $\gamma > 0$ is a tuning parameter, to be specified.  This is reminiscent to Zellner's $g$-prior \citep[e.g.,][]{zellner1986}, except that it is centered at the least squares estimator; see Section~\ref{SS:likelihood} for more on this data-dependent prior centering.  To summarize, our proposed prior $\Pi$ for $\beta$ is given by 
\begin{equation}
\label{eq:prior}
\Pi(d\beta) = \sum_{S: |S| \leq R} \nm_{|S|}\bigl( d\beta_S \bigmid \widehat\beta_S, \gamma^{-1}(X_S^\top X_S)^{-1} \bigr) \, \delta_0(d\beta_{S^c}) \, \pi(S). 
\end{equation}
Following \citet{martin.walker.eb}, we refer to this data-dependent prior as an empirical prior; see Section~\ref{SS:posterior}.  By restricting $|S| \leq R$, we can be sure that the least squares estimator $\widehat\beta_S$ is available, along with the usual distribution theory.  In our implementation, $\gamma^{-1}$ will be large, which means that the conditional prior for $\beta_S$ is rather diffuse, so the dependence on the data, through $\widehat\beta_S$, is not overly strong.

Obviously, to properly define the conditional prior for $\beta_S$, we implicitly assume that $X_S^\top X_S$ is non-singular for all subsets $S$ with $|S| \leq R$.  This is only for simplicity, however, since the theory in Section~\ref{S:theory} goes through without this assumption at the cost of making computations more difficult.

\subsection{The likelihood function}
\label{SS:likelihood}

For the likelihood function, write $L_n(\beta) = \nm_n(Y \mid X\beta, \sigma^2 I)$ as the $n$-dimensional Gaussian density at $Y$, with mean $X\beta$, covariance matrix proportional to the identity matrix, and treated as a function of $\beta$.  One unique feature of our approach so far is the centering of the (conditional) prior on the least squares estimator, which is greedy, in some sense.  To prevent the posterior from tracking the data too closely, the second feature of our proposed approach is that we introduce a fractional power $\alpha \in (0,1)$ on the likelihood.  That is, instead of $L_n(\beta)$, our likelihood will be $L_n(\beta)^\alpha$; see \citet{martin.walker.eb}.  Other authors have advocated the use of a fractional likelihood, including \citet{barron.cover.1991}, \citet{walker.hjort.2001}, \citet{zhang2006}, \citet{jiang.tanner.2008}, \citet{dalalyan.tsybakov.2008}, and \citet{grunwald.ommen.2014}, but these papers have different foci and none include a data-dependent (conditional) prior centering.  In fact, we feel that this combination of centering and fractional likelihood regularization (see Section~\ref{SS:posterior}) is a powerful tool that can be used for a variety of high-dimensional problems.  

Our analysis in what follows does not go through for the genuine Bayes case, corresponding to $\alpha=1$, but $\alpha$ can be arbitrarily close to 1.  Clearly, for finite-samples, the numerical differences between results for $\alpha \approx 1$ and for $\alpha=1$ are negligible.  

\subsection{The posterior distribution}
\label{SS:posterior}

Given the prior $\Pi$ for $\beta$ and the fractional likelihood, we form an empirical Bayes posterior distribution, denoted by $\Pi^n$, for $\beta$ using the standard Bayesian update.  That is, for $B$ a measurable subset of $\RR^p$, we have 
\begin{equation}
\label{eq:post0}
\Pi^n(B) = \frac{\int_B L_n(\beta)^\alpha \,\Pi(d\beta)}{\int_{\RR^p} L_n(\beta)^\alpha \, \Pi(d\beta)}. 
\end{equation}
Computation of this empirical Bayes posterior will be discussed in Section~\ref{S:numerical}.

We interpret ``empirical Bayes'' loosely---if the prior depends on data, then the corresponding posterior is empirical Bayes.  The combination of a prior, data-dependent or not, with a fractional likelihood via Bayes formula can also be understood from this empirical Bayes point of view.  Indeed,  
\[ L_n(\beta)^\alpha \, \Pi(d\beta) = L_n(\beta) \, \frac{\Pi(d\beta)}{L_n(\beta)^{1-\alpha}}, \]
i.e., the Bayes combination of a fractional likelihood with a prior is equivalent to a Bayes combination of the original likelihood function with a data-dependent prior.  As \citet{walker.hjort.2001} explain, rescaling the prior by a portion of the likelihood helps to protect from possible inconsistencies by penalizing those parameter values that ``track the data too closely.''  {Our proposal is obviously very different from the traditional empirical Bayes approach.  As stated in Section~\ref{S:intro}, our goal is simply to construct a data-dependent distribution for $\beta$ that is easy to compute and also has optimal concentration properties.  As a guide, we have followed the familiar prior-to-posterior updating, but added a new twist, and we will demonstrate in Sections~\ref{S:theory}--\ref{S:numerical} that our proposed empirical Bayes posterior distribution \eqref{eq:post0} does, indeed, achieve the stated objectives. }

\section{Posterior concentration rates}
\label{S:theory}

\subsection{Setup}
\label{SS:setup}

Before getting into details about the concentration rates, we first want to clarify what is meant by asymptotics in this context.  There is an implicit triangular array setup, i.e., for each $n$, the response vector $Y^n = (Y_1^n,\ldots,Y_n^n)^\top$ is modeled according to \eqref{eq:reg} with the $n \times p$ design matrix $X^n = ((X_{ij}^n))$, of rank $R \leq n$, which we take to be deterministic but depending on $n$, and vector of coefficients $\beta^n = (\beta_1,\ldots,\beta_p)^\top$.  When $n$ is increased, more data is available so, even though there are more variables to contend with (since $p \gg n$), there is hope that something about the true $\beta^n$ can be learnt, provided that it is sufficiently sparse.  In what follows, we will use the standard notation in \eqref{eq:reg} which is less cumbersome but hides the triangular array formulation.  It is important to keep in mind, however, that, throughout our analysis, $p$, $R$, and $s^\star$ depend implicitly on $n$.  


We make some minimal standing assumptions. First, without loss of generality, we can assume that $s^\star \leq R \leq n \ll p$.  No other assumptions concerning $n$, $p$, $R$, and $s^\star$ will be required.  The results below also hold for all fixed tuning parameters $\alpha \in (0,1)$ and $\gamma > 0$; see Section~\ref{SS:implementation} for guidance on the practical choice of $(\alpha,\gamma)$.  For the design matrix $X$, there is a standing simplifying assumption that we shall make.  In particular, we assume that $X_S$ is full-rank for each $S$ satisfying $|S| \leq R$.  This assumption holds, for example, if $X$ satisfies the ``sparse Riesz condition with rank $n$'' discussed in \citet{zhang.huang.2008} and \citet{chen.chen.2008}.  It is possible, however, to remove this condition, but it requires a modification of the empirical Bayes model.  Indeed, if the prior $\pi$ for $S$ only puts positive mass on those $S$ such that $X_S$ is full-rank, and if $X_{S^\star}$ is full-rank, then the theoretical results presented below follow similarly.  The drawback for adjusting the prior for $S$ in this way is additional computational cost, i.e., the less-than-full-rank models must be identified and removed by zeroing out the prior mass.  We opt here to keep things simple by making the full-rank assumption.

\subsection{A preliminary result}
\label{SS:prelim}

Let $B$ be a generic event for $\beta \in \RR^p$.  Our empirical Bayes posterior probability of the event $B$ in \eqref{eq:post0} can be rewritten as 
\begin{equation}
\label{eq:post}
\Pi^n(B) = \frac{\int_B R_n(\beta, \beta^\star)^\alpha \, \Pi(d\beta)}{\int R_n(\beta, \beta^\star)^\alpha \,\Pi(d\beta)}, 
\end{equation}
where $R_n(\beta,\beta^\star) = L_n(\beta)/L_n(\beta^\star)$ is the likelihood ratio.  Let $D_n$ denote the denominator in the above display, i.e., $D_n = \int R_n(\beta,\beta^\star)^\alpha \, \Pi(d\beta)$.  The next result, which will be useful throughout our analysis, gives a sure lower bound on $D_n$.  

\begin{lem}
\label{lem:denominator}
There exists $c=c(\alpha, \gamma, \sigma^2) > 0$ such that $D_n \geq \pi(S^\star)e^{-c|S^\star|}$.  
\end{lem}

\begin{proof}
$D_n$ is an average of a non-negative $S$-dependent quantity with respect to $\pi(S)$.  This average is clearly greater than the quantity for $S=S^\star$ times $\pi(S^\star)$.  That is, 
\begin{align*}
D_n & > \pi(S^\star) \int R_n(\beta, \beta^\star)^\alpha \nm(\beta_{S^\star} \mid \hat\beta_{S^\star}, \gamma^{-1}(X_{S^\star}^\top X_{S^\star})^{-1}) \, d\beta_{S^\star} \\
& = \pi(S^\star) \int e^{-\frac{\alpha}{2\sigma^2}\{\|Y-X_{S^\star}\beta_{S^\star}\|_2^2 - \|Y-X_{S^\star}\beta_{S^\star}^\star\|_2^2\}} \nm(\beta_{S^\star} \mid \hat\beta_{S^\star}, \gamma^{-1}(X_{S^\star}^\top X_{S^\star})^{-1}) \,d\beta_{S^\star}.
\end{align*}
Direct calculation shows that the lower bound above equals 
\[ \pi(S^\star) e^{\frac{\alpha}{2\sigma^2}\|X_{S^\star}(\hat\beta_{S^\star}-\beta_{S^\star}^\star)\|_2^2} \Bigr(1 + \frac{\alpha}{\gamma \sigma^2} \Bigr)^{-|S^\star|/2}. \]
Using the trivial bound $\|\cdot\|_2 \geq 0$ on the norm in the exponent, the proof is complete if we let $c = \frac12\log\bigl( 1 + \frac{\alpha}{\gamma \sigma^2} \bigr)$, which is clearly positive.  
\end{proof}

\subsection{Prediction loss}
\label{SS:prediction}

We now present a result characterizing the concentration rate of the posterior distribution for the mean $X\beta$.  Set 
\begin{equation}
\label{eq:prediction.event}
B_{\eps_n} = \{\beta \in \RR^p: \|X(\beta-\beta^\star)\|_2^2 > \eps_n\}, 
\end{equation}
where $\eps_n$ is a positive sequence to be specified.  Since this loss involves the $X$ matrix, the notion of convergence we are considering here is related to prediction.  Different loss functions will be considered in Section~\ref{SS:other}.  As discussed in \citet{buhlmann.geer.book}, e.g., their equation (2.8), $\eps_n$ proportional to $s^\star \log p$ corresponds to the convergence rate for the lasso estimator.  Intuitively, if $S^\star$ were known, then the best rate for the prediction error would be $s^\star$, so the logarithmic term acts as a penalty for having to also deal with the unknown model.  

Let $N_n$ be the numerator for the posterior probability of $B_{\eps_n}$, as in \eqref{eq:post}, i.e., $N_n = \int_{B_{\eps_n}} R_n(\beta,\beta^\star)^\alpha \, \Pi(d\beta)$.  We have the following bound on $N_n$.      

\begin{lem}
\label{lem:numerator}
There exists $d=d(\alpha, \sigma^2) > 0$ and $\phi=\phi(\alpha, \gamma, \sigma^2) > 1$ such that $\E_{\beta^\star}(N_n) \leq e^{-d \eps_n} \sum_{S: |S| \leq R} \phi^{|S|} \pi(S)$, uniformly in $\beta^\star$.  
\end{lem}

\begin{proof}
See Appendix~\ref{SS:proof.num}.  
\end{proof}

\ifthenelse{1=1}{}{
\begin{remark}
\label{re:hyper}
{\color{red} Following the intuition in Section~\ref{SS:prior}, we take $\alpha$ close to 1.  This makes the constant $d$ in Lemma~\ref{lem:numerator} small.  Also, if $\alpha$ is close to 1, then $\gamma$ can be chosen small enough so that the constant $\phi$ is not too big, i.e., close to $2^{1/2}$.  Justification for these statements can be found in the proof of Lemma~\ref{lem:numerator}; see, also, Section~\ref{S:numerical}.  } 
\end{remark}
}

To bound the posterior probability of $B_{\eps_n}$, let $b_n = \pi(S^\star) e^{-c s^\star}$.  Since $D_n \geq b_n$, surely, by Lemma~\ref{lem:denominator}, we have 
\[ \Pi^n(B_{\eps_n}) = \frac{N_n}{D_n} \cdot 1(D_n \geq b_n) + \frac{N_n}{D_n} \cdot 1(D_n < b_n) \leq \frac{N_n}{b_n}. \]
Taking expectation and plugging in the bound in Lemma~\ref{lem:numerator} gives 
\begin{align*}
\E_{\beta^\star}\{\Pi^n(B_{\eps_n})\} & \leq e^{c|S^\star|-d\eps_n} \, \frac{1}{\pi(S^\star)} \sum_S \phi^{|S|} \pi(S) \\
& = e^{cs^\star-d\eps_n} \, \frac{\binom{p}{s^\star}}{f_n(s^\star)} \sum_{s=0}^R \phi^s f_n(s); 
\end{align*}
which holds uniformly in $\beta^\star$ with $|S_{\beta^\star}|=s^\star$.  Then the empirical Bayes concentration rate $\eps_n = \eps_n(p, R, s^\star)$ is such that the above upper bound vanishes.  A first conclusion is that $\eps_n$ must satisfy $s^\star = o(\eps_n)$.  More precisely, if we set
\[ \zeta_n = \zeta_n(p, R, s^\star) = \frac{\binom{p}{s^\star}}{f_n(s^\star)} \sum_{s=0}^R \phi^s f_n(s), \]
then the rate $\eps_n$ satisfies 
\begin{equation}
\label{eq:growth}
\log\zeta_n = O(\eps_n), \quad \text{as $n \to \infty$}.
\end{equation}
This amounts to a condition on the prior $f_n$ for $|S|$.  Indeed, \eqref{eq:growth} requires that $f_n$ should be sufficiently concentrated near $s^\star$, so that $f_n(s^\star)$ is not too small and the expectation of $\phi^{|S|}$ with respect to $f_n$ is not too big.  Compare this to the prior support conditions in \citet{ggv2000}, \citet{shen.wasserman.2001}, and \citet{walker2007}.    

We are now ready to state and prove our first main concentration rate result.  To keep the statement of the theorem concise, we give an asymptotic convergence result.  However, Theorem~\ref{thm:mean.concentration} and Theorems~\ref{thm:dim}--\ref{thm:selection2} in the upcoming sections, are actually stronger than stated, since the proofs are based on getting explicit fixed-$(n,p,s^\star)$ bounds.

\begin{thm}
\label{thm:mean.concentration}
For any $s^\star \leq R$, if the prior $f_n$ on $|S|$ admits $\zeta_n$ such that \eqref{eq:growth} holds with $\eps_n$, then there exists a constant $M > 0$ such that $\E_{\beta^\star} \{\Pi^n(B_{M\eps_n})\} \to 0$ as $n \to \infty$, uniformly in $\beta^\star$ with $|S_{\beta^\star}| = s^\star$.  
\end{thm}

\begin{proof}
By Lemmas~\ref{lem:denominator} and \ref{lem:numerator}, and the growth condition \eqref{eq:growth}, we have that, for large $n$, 
\[ \log \E_{\beta^\star}\{\Pi^n(B_{M\eps_n})\} \leq \Bigl( \frac{cs^\star}{\eps_n} - Md + \frac{\log\zeta_n}{\eps_n} \Bigr) \eps_n. \]
The first term inside the parentheses vanishes since $s^\star = o(\eps_n)$.  Next, under \eqref{eq:growth}, there exists a $K > 0$ such that $(\log\zeta_n)/\eps_n < K$.  So, if we take $M$ such that $Md > K$, then the upper bound above goes to $-\infty$ as $n \to \infty$.  This implies the result.   
\end{proof}

\begin{remark}
\label{re:minimax}
What rates $\eps_n$ are desirable/attainable?  The minimax rate for estimation under this prediction error loss is $\min\{R, s^\star \log(p / s^\star)\}$; see, e.g., \citet{rigollet.tsybakov.2012}.  Note the phase transition between the ordinary [$s^\star \log(p / s^\star) < R$] and the ultra high-dimensional [$s^\star \log(p / s^\star) > R$] regimes. According to Remark~\ref{re:geometric}, an empirical Bayes posterior concentration rate equal to $s^\star \log(p / s^\star)$ obtains for a class of priors on $S$, which is minimax optimal but only in the ordinary high-dimensional regime; this rate is slightly better than those obtained in \citet{ariascastro.lounici.2014} and \citet{castillo.schmidt.vaart.2014}, but see \citet[][Corollary~5.3]{gao.vdv.zhou.2015} for a result comparable to ours in Theorem~\ref{thm:mean.concentration}.  By picking a prior outside this class, in particular, one that puts a little mass on an overly-complex model, the minimax rate can be achieved in both the ordinary and ultra high-dimensional regimes.  There is a price to be paid, however, for this complete minimax rate: the little piece of extra prior mass on the complex model is large  enough to cause problems with the proofs of marginal posterior concentration properties for $S$.  Justification of these claims can be found in {Appendix~\ref{S:claims}}.  Based on these observations, we conjecture that the priors on $S$ that lead to minimax concentration rate under prediction error loss do not lead to desirable model selection properties.  This is intuitively reasonable, since good prediction generally does not require a correctly specified model, but more work is needed to confirm this.  Since we prefer to have a single prior that does well in all aspects, we will not concern ourselves here with attaining the optimal minimax rate in the ultra high-dimensional regime, though we do know how to obtain it.  
\end{remark}

\begin{remark}
\label{re:geometric}
The growth condition \eqref{eq:growth} holds with $\eps_n$ proportional to $s^\star \log(p / s^\star)$, the minimax rate in the ordinary high-dimensional case, if there exists constants $a_1$, $a_2$, $c_1$, $c_2$, $C_1$, and $C_2$ such that $f_n$ satisfies 
\begin{equation}
\label{eq:prior.dim}
C_1\Bigl( \frac{1}{c_1 p^{a_1}} \Bigr)^s \leq f_n(s) \leq C_2 \Bigl( \frac{1}{c_2 p^{a_2}} \Bigr)^s \quad \text{for all $s=0,1,\ldots,R$} 
\end{equation}
The proof of this claim follows from calculations similar to those in Example~\ref{ex:complexity} below.  Assumption~1 in \citet{castillo.schmidt.vaart.2014} implies \eqref{eq:prior.dim}, but our restriction, $|S| \leq R$, allows us to get rates for priors that may not satisfy \eqref{eq:prior.dim}.
\end{remark}

\begin{remark}
\label{re:phi}
Consider the expectation term $\sum_{s=0}^R \phi^s f_n(s)$.  The trivial bound $\phi^R$ could be used in the ultra high-dimensional case where $s^\star \log(p / s^\star) \gg R$.  More generally, if $f_n$ satisfies \eqref{eq:prior.dim}, then the formulas for partial sums of a geometric series reveal that this expectation term is bounded as $n \to \infty$.  In fact, in the examples discussed below, it is easy to confirm that the expectation term is bounded.  Therefore, the rate is determined completely by the prior concentration around $S^\star$. 
\end{remark}
  
Next we identify the rate $\eps_n$ corresponding to several choices of prior $f_n$.  The complexity prior in Example~\ref{ex:complexity}, which is simple and has good properties,  will be our choice of prior in what follows; our proofs Sections~\ref{SS:dimension}--\ref{SS:selection} can be easily modified to cover any $f_n$ that satisfies \eqref{eq:prior.dim}.  

\begin{ex}
\label{ex:complexity}
The complexity prior for the model size $|S|$ in Equation~(2.3) of \citet{castillo.schmidt.vaart.2014} is given by 
\begin{equation}
\label{eq:complexity}
f_n(s) \propto c^{-s} p^{-as}, \quad s=0,1,\ldots,R, 
\end{equation}
where $a$ and $c$ are positive constants.  This prior clearly satisfies the condition \eqref{eq:prior.dim} in Remark~\ref{re:geometric}.  We claim that this complexity prior satisfies \eqref{eq:growth} with $\eps_n = s^\star\log(p / s^\star)$.  To see this, note that $\log f_n(s^\star)$ is lower bounded by
\[ -s^\star \log(c s^{\star a}) -as^\star\log(p / s^\star) = -\Bigl( a + \frac{\log c + a \log s^\star}{\log(p/s^\star)} \Bigr) s^\star \log(p / s^\star). \]
The ratio inside the parentheses above vanishes since $s^\star \ll p$.  Similarly, by Stirling's formula, we have that $\log \binom{p}{s^\star} \leq s^\star \log(p / s^\star)\{1 + o(1)\}$.  Putting these two bounds together, and using the result in Remark~\ref{re:phi}, we can conclude that the complexity prior above yields a posterior concentration rate $s^\star \log(p / s^\star)$.  
\end{ex}

\begin{ex}
\label{ex:beta.binom}
Convergence rates can be obtained for other priors $f_n$.  First, consider a beta--binomial prior for $|S|$, i.e., 
\[ f_n(s) = \int_0^1 \binom{R}{s} w^{R-s}(1-w)^s \, a_n w^{a_n-1} \,dw, \]
which corresponds to a $\bet(a_n,1)$ prior for $W$ and a conditional $\bin(R,1-w)$ prior for $|S|$, given $W=w$.  For $a_n = a R$, for a constant $a > 0$, it can be shown that the corresponding rate $\eps_n$ is proportional to $s^\star \log(p / s^\star)$.  If, on the other hand, $f_n$ is a $\bin(R, R^{-1})$ mass function, then similar calculations show that the concentration rate is $\eps_n = s^\star \log p$, which agrees with the lasso rate in \citet{buhlmann.geer.book}, but falls short of the rates discussed previously. 
\end{ex}

\ifthenelse{1=1}{}{
There are other priors whose corresponding posterior concentration rate is worse than $s^\star \log(p / s^\star)$, if only by a little bit.  The next example illustrates this.  

\begin{ex}
\label{ex:binomial}
Let $f_n$ be a $\bin(R,R^{-1})$ mass function.  Using the same lower and upper bounds on the binomial coefficients as in the previous example, we have 
\begin{align*}
\frac{f_n(s^\star)}{\binom{p}{s^\star}} 
& \geq p^{-s^\star} e^{-s^\star-1} \{1 + o(1)\}. 
\end{align*}
Again, the expectation of $\phi^{|S|}$ is bounded as $n \to \infty$, so we get that $\log\zeta_n = O(s^\star\log p)$.  This matches the rate for the lasso quoted in Equation~(2.8) of \citet{buhlmann.geer.book}, but falls short of the rates in the previous examples.    
\end{ex}
}

\subsection{Effective dimension}
\label{SS:dimension}

Under our proposed prior, the empirical Bayes posterior distribution for $\beta$ is concentrated on an $R$-dimensional subspace of the full $p$-dimensional parameter space.  In the sparse case, where the true $\beta^\star$ has effective dimension $s^\star \leq R\ll p$, it is interesting to ask if the posterior distribution is actually concentrated on a space of dimension close to  $s^\star$.  Below we give an affirmative answer to this question under some conditions.  Such considerations will also be useful in Sections~\ref{SS:other} and \ref{SS:selection}.   

For a given $\Delta$, let $B_n(\Delta) = \{\beta \in \RR^p: |S_\beta| \geq \Delta\}$ be those $\beta$ vectors with no less than $\Delta$ non-zero entries.  We say that the effective dimension of $\Pi^n$ is bounded by $\Delta=\Delta_n$ if the expected posterior probability of $B_n(\Delta)$ vanishes as $n \to \infty$.  Next write 
\[ N_n(\Delta) = \int_{B_n(\Delta)} R_n(\beta,\beta^\star)^\alpha \, \Pi(d\beta), \]
for the numerator of the posterior probability of $B_n(\Delta)$.  

\begin{lem}
\label{lem:dim}
$\E_{\beta^\star}\{N_n(\Delta)\} \leq \sum_{s = \Delta}^R \phi^s f_n(s)$ for all $\beta^\star$. 
\end{lem}

\begin{proof}
See Appendix~\ref{SS:proof.dim}.  
\end{proof}

We can combine Lemma~\ref{lem:dim} and Lemma~\ref{lem:denominator} to conclude that
\[ \E_{\beta^\star}[ \Pi^n\{B_n(\Delta)\} ] \leq e^{cs^\star} \frac{\binom{p}{s^\star}}{f_n(s^\star)} \sum_{s = \Delta}^R \phi^s f_n(s), \]
uniformly in $\beta^\star$ with $|S_{\beta^\star}| = s^\star$.  Since $\phi > 1$, we have $\sum_s \phi^s f_n(s) > 1$ and, therefore, 
\begin{equation}
\label{eq:dim.bound}
\E_{\beta^\star}[ \Pi^n\{B_n(\Delta)\} ] \leq e^{cs^\star + \log \zeta_n} \sum_{s=\Delta}^R \phi^s f_n(s). 
\end{equation}
So, if the tail of the prior $f_n$ on the model size is sufficiently light, then the posterior probability assigned to models with complexity of order greater than $s^\star$ will be small.  Under the conditions of Theorem~\ref{thm:mean.concentration}, we know the magnitude of $\log\zeta_n$, but here we need additional control on the tails of $f_n$.  

\begin{thm}
\label{thm:dim}
Let $s^\star \leq R$.  If $f_n$ is of the form \eqref{eq:complexity}, then $\E_{\beta^\star}[\Pi^n\{B_n(\Delta_n)\}] \to 0$, holds with $\Delta_n = Cs^\star$, uniformly in $\beta^\star$ with $|S_{\beta^\star}|=s^\star$, i.e., the effective dimension $\Pi^n$ is bounded by $Cs^\star$.  
\end{thm}

\begin{proof}
Recall that, for this $f_n$, $\log\zeta_n$ is of the order $s^\star \log(p/s^\star)$.  Moreover, for a generic $\Delta$, the summation $\sum_{s=\Delta}^R \phi^s f_n(s)$ is bounded by a partial sum of a geometric series.  In particular, the bound is $O(r^{\Delta+1})$, where $r = \phi/cp^a$ and $a,c$ are in \eqref{eq:complexity}.  In that case, 
\[ r^{\Delta+1} = e^{-(\Delta+1)[a\log p + \log(c/\phi)]}. \]
So, if $\Delta$ is a suitable multiple of $s^\star$, then clearly the $r^{\Delta+1}$ term dominates the $e^{cs^\star + \log\zeta_n}$ term.  In particular, if $\Delta = C s^\star$ with $C > a^{-1}$, then the product on the right-hand side of \eqref{eq:dim.bound} vanishes, proving the claim.
\end{proof}


To summarize, our prior is such that the posterior distribution is supported on models of size no more than $R$.  However, a good prior is one such that the posterior ought to be able to learn the size of the true model that generated the data, which is possibly much less than $R$.  Theorem~\ref{thm:dim} shows that, indeed, if the prior $f_n$ on the model size has sufficiently light tails, controlled by the prior exponent $a > 0$, then the posterior will concentrate on models of size proportional to $s^\star$, the true model size.  We cannot take a $\tilde{s}<s^\star$ to replace $s^\star$ in $Cs^\star$, since we would need
\[ \frac{\tilde{s}}{s^\star} \geq \frac{\log(p/s^\star)}{\log p} \to 1, \]
which confirms this particular point.  Furthermore, we see exactly the effect that the prior exponent $a$ has through the bound $C > a^{-1}$ on the proportionality constant.  So, small $a$ will have the effect of spreading out the posterior to include some large (but not too large) models, while large $a$ will keep the posterior concentrated on small models.  Choosing small $a$ is beneficial in finite-sample studies; see Section~\ref{S:numerical}.

\subsection{Other loss functions}
\label{SS:other}

Theorem~\ref{thm:mean.concentration} concerns the empirical Bayes posterior probability of sets of $\beta$ which are near the true $\beta^\star$ relative to a distance depending on the design matrix $X$.  A natural question is if the empirical Bayes posterior concentrates on neighborhoods of $\beta^\star$ with respect to more other metrics, such as $\ell_1$- and $\ell_2$-norms.  An affirmative answer will require further conditions on $X$ to separate $\beta$ from $X\beta$.  

In the low-dimensional case, with $p < n$, we have
\[ \|X(\beta-\beta^\star)\|_2 \geq \lambda_{\text{min}}(X^\top X)^{1/2} \|\beta-\beta^\star\|_2, \] 
where $\lambda_{\text{min}}(A)$ is the minimum eigenvalue of $A$, which is positive if $A$ is non-singular.  When $p \gg n$, $X$ is not full rank and, therefore, the smallest eigenvalue of $X^\top X$ is zero, making the above inequality trivial and not useful.  However, it is still possible to get something like the displayed inequality.  Towards this, define the function
\begin{equation}
\label{eq:kappa}
\kappa(s) = \kappa_X(s) = \inf_{\beta: 0 < |S_\beta| \leq s} \frac{\|X\beta\|_2}{\|\beta\|_2}, \quad s=1,\ldots,p.
\end{equation}
The quantity $\kappa(s)$ is called the ``smallest scaled sparse singular value of $X$ of dimension $s$,'' similar to the quantity in Equation~(11) of \citet{ariascastro.lounici.2014} and that in Definition~2.3 of \citet{castillo.schmidt.vaart.2014}.  Its main purpose is to facilitate conversion of $\ell_2$-norm concentration results for the mean vector $X\beta$ to $\ell_2$-norm concentration results for $\beta$ itself.  Indeed, a result shown in \citet[][Lemma~1]{ariascastro.lounici.2014} is that a true $\beta^\star \in \RR^p$ with $|S_{\beta^\star}|=s^\star$ is identifiable if and only if 
\begin{equation}
\label{eq:identifiable}
\kappa(2s^\star) > 0. 
\end{equation}
Consequently, $\kappa$ is an important quantity and will appear in Theorem~\ref{thm:beta.concentration} below. One can define quantities analogous to $\kappa$ in order to get concentration results relative to the $\ell_1$- or $\ell_\infty$-norm of $\beta$; see \citet[][Section~2]{castillo.schmidt.vaart.2014}.  

The result presented below will follow almost immediately from Theorem~\ref{thm:mean.concentration} and the definition of $\kappa$.  Indeed, for any $\beta$, we have
\begin{equation}
\label{eq:L2.bound}
\|X(\beta-\beta^\star)\|_2 \geq \kappa(|S_{\beta-\beta^\star}|) \, \|\beta-\beta^\star\|_2. 
\end{equation}
For example, if $\|\beta-\beta^\star\|_2$ is lower-bounded, then so is $\|X(\beta-\beta^\star)\|_2$, for suitable $\kappa$, so a posterior concentration result for the $\ell_2$-norm on $\beta$ should follow from an analogous result for the $\ell_2$ prediction error as in Theorem~\ref{thm:mean.concentration}.  The only obstacle is that the $\kappa$ term on the right-hand depends on the particular $\beta$.  The following result leads to the observation that $\kappa(|S_{\beta-\beta^\star}|)$ can be controlled by a term that depends only on $s^\star$.   

\begin{lem}
\label{lem:kappa}
For any $\beta$ and $\beta^\star$, $\kappa(|S_{\beta-\beta^\star}|) \geq \kappa(|S_\beta| + |S_{\beta^\star}|)$.
\end{lem}

\begin{proof}
This follows since $\kappa$ is non-increasing and $|S_{\beta-\beta^\star}| \leq |S_\beta| + |S_{\beta^\star}|$.
\end{proof}

Under our prior formulation, we know that the posterior puts probability~1 on those $\beta$ for which $|S_\beta| \leq R$.  So, if $|S_{\beta^\star}| = s^\star$, then, trivially, $\kappa(|S_{\beta-\beta^\star}|) \geq \kappa(R + s^\star)$.  For better control on the $\kappa$ term in \eqref{eq:L2.bound}, recall that Theorem~\ref{thm:dim} says that the posterior probability of the event $\{|S_\beta| \geq Cs^\star\}$ vanishes as $n \to \infty$.  Therefore, for $C'=C+1$, 
\begin{equation}
\label{eq:kappa.bound}
\kappa(|S_{\beta-\beta^\star}|) \geq \kappa(C's^\star) 
\end{equation}
holds for all $\beta$ in a set with posterior probability approaching 1.  Compare this to Theorem~1 of \citet{castillo.schmidt.vaart.2014}, and also to the corresponding model selection results for frequentist point estimators in, e.g., \citet[][Chap.~7]{buhlmann.geer.book}.  

We are now ready for the concentration rate result with respect to the $\ell_2$-norm loss on the parameter $\beta$ itself.  This time, set 
\[ B_{\delta_n}' = \{\beta \in \RR^p: \|\beta-\beta^\star\|_2^2 > \delta_n\}, \]
where $\delta_n$ is a positive sequence to be specified.  

\begin{thm}
\label{thm:beta.concentration}
For $s^\star \leq \min(n, R)$, suppose that the prior $f_n$ satisfies \eqref{eq:complexity} with exponent $a > 0$, so that Theorem~\ref{thm:mean.concentration} holds with $\eps_n$ equal to $s^\star\log(p/s^\star)$ and Theorem~\ref{thm:dim} holds with $\Delta_n=Cs^\star$, for $C > a^{-1}$.  Then there exists a constant $M$ such that $\E_{\beta^\star}\{\Pi^n(B_{M\delta_n}')\} \to 0$ as $n \to \infty$, uniformly in $\beta^\star$ with $|S_{\beta^\star}|=s^\star$, where 
\[ \delta_n = \frac{s^\star \log(p / s^\star)}{\kappa(C's^\star)^2}, \]
provided that $\kappa(C's^\star) > 0$, where $C' = 1+C > 1 + a^{-1} > 1$.  
\end{thm}

\begin{proof}
It follows immediately from \eqref{eq:L2.bound} that $\|\beta-\beta^\star\|_2^2 > M\delta_n$ implies 
\[ \|X(\beta-\beta^\star)\|_2^2 > M\kappa(|S_{\beta-\beta^\star}|)^2 \delta_n. \]
By definition of $\delta_n$ and the inequality \eqref{eq:kappa.bound}, this last inequality implies 
\[ \|X(\beta-\beta^\star)\|_2^2 > M s^\star \log(p / s^\star). \]
If we take $M$ as in Theorem~\ref{thm:mean.concentration}, then the event in the above display is exactly $B_{M\eps_n}$.  We have shown that $\Pi^n(B_{M\delta_n}') \leq \Pi^n(B_{M\eps})$.  By Theorem~\ref{thm:mean.concentration}, the expectation of the upper bound vanishes uniformly in $\beta^\star$ as $n \to \infty$, so the proof is almost complete.  The remaining issue to deal with is an extra term in the upper bound for $\Pi^n(B_{M\delta_n}')$ coming from using $\kappa(C's^\star)$ in place of $\kappa(|S_{\beta-\beta^\star}|)$ above.  However, this extra term is $o(1)$ by Theorem~\ref{thm:dim}, and, therefore, does not actually impact the proof.
\end{proof}

Compare this result to the third in Theorem~2 of \citet{castillo.schmidt.vaart.2014}.  First, our rate is slightly better, $s^\star \log(p / s^\star)$ compared to the lasso rate $s^\star \log p$.  Second, our bound does not depend on a ``compatibility number'' \citep[e.g.,][Definition~2.1]{castillo.schmidt.vaart.2014}, which also improves the rate and makes interpretation of our result easier.  {A referee has indicated that the improved results are as a direct consequence of the $(X_S^\top X_S)^{-1}$ term that appears in the prior for $\beta_S$.}  Also, the condition $\kappa(C' s^\star) > 0$, with $C' = 1 + a^{-1}$ and $a < 1$, agrees with the condition, roughly, $\kappa\bigl( (2 + \eps)s^\star \bigr) > 0$ for some $\eps > 0$, in \citet{ariascastro.lounici.2014}; that is, just a little more than identifiability, as in \eqref{eq:identifiable} is needed.

\subsection{Model selection}
\label{SS:selection}

Interest here is on the model $S$ and not directly on the regression coefficients.  In this case, it is convenient to work with the marginal posterior distribution for $S$ which, thanks to the simple conjugate structure in the conditional prior, we can write explicitly as 
\begin{equation}
\label{eq:marginal0}
\pi^n(S) \propto \pi(S) e^{-\frac{\alpha}{2\sigma^2}\|Y-\hat Y_S\|^2} \nu^{-|S|}, 
\end{equation}
where $\nu=(\gamma + \alpha / \sigma^2)^{1/2}$.  Then 
\begin{equation}
\label{eq:model.bound}
\pi^n(S) \leq \frac{\pi^n(S)}{\pi^n(S^\star)} = \frac{\pi(S)}{\pi(S^\star)} \nu^{|S^\star|-|S|} e^{\frac{\alpha}{2\sigma^2}\{\|Y-\hat Y_{S^\star}\|^2 - \|Y-\hat Y_S\|^2\}}. 
\end{equation}
From this bound, we can show that the posterior concentrates on models contained in $S^\star$, i.e., asymptotically, it will not charge any models with unnecessary variables.  Furthermore, this conclusion requires no conditions on the $X$ matrix or true $\beta^\star$.  For simplicity, we will focus on the particular complexity prior $f_n$ in \eqref{eq:complexity} shown previously to yield desirable posterior concentration properties.  

\begin{thm}
\label{thm:selection1}
Let the constant $a > 0$ in the complexity prior \eqref{eq:complexity} be such that $p^a \gg R$.  Then $\E_{\beta^\star}\{\Pi^n(\beta: S_\beta \supset S_{\beta^\star})\} \to 0$, uniformly over $\beta^\star$.  
\end{thm}

\begin{proof}
See Appendix~\ref{proof:selection1}.  
\end{proof}

Theorem~\ref{thm:selection1} says that, asymptotically, our empirical Bayes posterior will not include any unnecessary variables.  It remains to say what it takes for the posterior to asymptotically identify all the important variables.  The first condition is one on the $X$ matrix, specifically, if $s^\star$ is the true model size, then we require $\kappa(s^\star) > 0$; this is implied by monotonicity of $\kappa$ and the identifiability condition \eqref{eq:identifiable} in Section~\ref{SS:other}.  For our second assumption, we consider the magnitudes of the non-zero entries in a $s^\star$-sparse $\beta^\star$.  Intuitively, we cannot hope to be able to distinguish between an actual zero and a very small non-zero, but defining what is ``very small'' requires some care.  Here, we define this cutoff by 
\begin{equation}
\label{eq:rho}
\rho_n = \frac{\sigma}{\kappa(s^\star)} \Bigl\{ \frac{2M(1+\alpha)}{\alpha} \, \log p \Bigr\}^{1/2}, 
\end{equation}
where $M > 0$ is a constant to be determined.  In particular, coefficients of magnitude greater than $\rho_n$ are large enough to be detected.  The so-called \emph{beta-min} condition assumes that all the non-zero coefficients are sufficiently far from zero.  The cutoff $\rho_n$ in \eqref{eq:rho} is better than that appearing in Equation~(2.18) in \citet{buhlmann.geer.book} for the lasso model selector but comparable to that in Theorem~1 of \citet{ariascastro.lounici.2014} and in the third part of Theorem~5 in \citet{castillo.schmidt.vaart.2014}, where the latter requires additional assumptions on $X$.  

\begin{thm}
\label{thm:selection2}
For any $s^\star \leq R$, let $\beta^\star$ be such that $|S_{\beta^\star}| = s^\star$ and 
\[ \min_{j \in S^\star} |\beta_j^\star| \geq \rho_n, \]
with $M > a+1$, where $a > 0$ is in the complexity prior, with $p^a \gg R$. Assuming the condition of Theorem~\ref{thm:dim} holds,  if $\kappa(s^\star) > 0$, then $\E_{\beta^\star}\{\Pi^n(\beta: S_\beta = S_{\beta^\star})\} \to 1$.  
\end{thm}

\begin{proof}
See Appendix~\ref{proof:selection2}.
\end{proof}

\section{Numerical results}
\label{S:numerical}

\subsection{Implementation}
\label{SS:implementation}

To compute our empirical Bayes posterior distribution, we employ a Markov chain Monte Carlo method.  To start, recall from \eqref{eq:marginal0} that we can write the marginal posterior mass function, $\pi^n(S)$, for the model $S$ can be written down explicitly, i.e., 
\[ \pi^n(S) \propto \pi(S) \, e^{-\frac{\alpha}{2\sigma^2}\|Y - \widehat Y_S\|^2} \Bigl(\gamma + \frac{\alpha}{\sigma^2} \Bigr)^{-|S|/2}, \]
where $\widehat Y_S = X_S \widehat\beta_S$ is the least-squares prediction for model $S$.  Intuitively, there are three contributing factors to the posterior distribution for $S$, namely, the prior probability of the model, a measure of how well the model fits the data, and an additional penalty on the complexity of the model.  So, clearly, the posterior distribution will favor models with smaller number of variables that provide adequate fit to the observed $Y$.  This provides further intuition about theorems presented in Section~\ref{S:theory}.  

Besides this intuition, the formula $\pi^n(S)$ provides a convenient way to run a Rao--Blackwellized Metropolis--Hastings method to sample from the posterior distribution of $S$.  Indeed, if $q(S' \mid S)$ is a proposal function, then a single iteration of our proposed Metropolis--Hastings sampler goes as follows:
\begin{enumerate}
\item Given a current state $S$, sample $S' \sim q(\cdot \mid S)$.
\vspace{-2mm}
\item Move to the new state $S'$ with probability 
\[ \min\Bigl\{1, \frac{\pi^n(S')}{\pi^n(S)} \frac{q(S \mid S')}{q(S' \mid S)} \Bigr\}; \]
otherwise, stay at state $S$.
\end{enumerate}
Repeating this process $M$ times, we obtain a sample of models $S_1,\ldots,S_M$ from the posterior $\pi^n(S)$.  Monte Carlo approximations of, say, the inclusion probabilities (Section~\ref{SS:simulations}) of individual variables can then easily be computed based on this sample.  In our case, we use a symmetric proposal distribution $q(S' \mid S)$, i.e., one that samples $S'$ uniformly from those models that differ from $S$ in exactly one position, which simplifies the acceptance probability above since the $q$-ratio is identically 1.  Also, we initialize our Markov chain Monte Carlo search at the model selected by lasso.  

To implement this procedure, some additional tuning parameters need to be specified.  First, recall that $(\alpha,\gamma) = (1,0)$ corresponds to the genuine Bayes model with a flat prior for $\beta_S$.  Our theory does not cover this case, but we can mimic it by picking something close.  Here we consider $\alpha=0.999$ and $\gamma=0.001$; in our experience, the performance is not sensitive to the choice of $(\alpha,\gamma)$ in a neighborhood of $(0.999, 0.001)$.  Second, for the prior on the model size, we employ the complexity prior \eqref{eq:complexity} with $c=1$ and $a=0.05$, i.e., $f_n(s) \propto p^{-0.05 s}$.  The choice of small $a$ makes the prior sufficiently spread out, allowing the posterior to move across the model space and, in particular, helping the Markov chain for $S$ to mix reasonably well.  Third, in practice, the error variance $\sigma^2$ is seldom known, so some procedure to handle unknown $\sigma^2$ is needed.  We proposed to modify our empirical Bayes posterior by plugging in an estimate of $\sigma^2$.  In particular, we use a residual mean square error based on a lasso fit \citep{reid.tibshirani.friedman.2014}.  


Finally, if samples from the $\beta$ posterior are desired, then these can easily be obtained, via conjugacy, after a sample of $S$ is available.  In particular, the conditional posterior distribution for $\beta_S$, given $S$, is normal with mean $\hat\beta_S$ and variance $(\gamma + \frac{\alpha}{\sigma^2})^{-1} (X_S^\top X_S)^{-1}$.  R code to implement our procedure is available at \url{www.math.uic.edu/~rgmartin}.

\subsection{Simulations}
\label{SS:simulations}

In this section, we reconsider some of the simulation experiments performed by \citet{narisetty.he.2014}, which are related to experiments presented in \citet{johnson.rossell.2012}.  In each setting, the error variance is $\sigma^2=1$; the covariate matrix is obtained by sampling from a multivariate normal distribution with zero mean, unit variance, and constant pairwise correlation $\rho=0.25$; and the true model $S^\star$ has $s^\star=5$.  The particular correlation structure among the covariates is given practical justification in \citet{johnson.rossell.2012}.  Under this setup, we consider three different settings:
\begin{description}
\item[\it Setting~1.] $n=100$, $p=500$, and $\beta_{S^\star}=(0.6, 1.2, 1.8, 2.4, 3.0)^\top$;
\vspace{-2mm}
\item[\it Setting~2.] $n=200$, $p=1000$, and $\beta_{S^\star}$ the same as in Setting~1;
\vspace{-2mm}   
\item[\it Setting~3.] $n=100$, $p=500$, and $\beta_{S^\star}=(0.6, 0.6, 0.6, 0.6, 0.6)^\top$.
\end{description}
Our Settings~1--2 correspond to the two $(n,p)$ configurations in Case~2 of \citet{narisetty.he.2014} and our Setting~3 is the same as their Case~3.  

We carry out model selection by retaining those variables whose inclusion probability $p_j = \Pi^n(\beta_j \neq 0)$, $j=1,\ldots,p$, exceeds 0.5; this is the so-called median probability model, shown to be optimal, in a certain sense, by \citet{barbieri.berger.2004}.  Alternatively, one could select the model with the largest posterior probability, but this is more expensive computationally compared to the median probability model---only $p$ inclusion probabilities instead of up to $2^p$ model probabilities.  In all cases, the posterior almost immediately concentrates on the true model.  Our Markov chain Monte Carlo required only 5000 iterations to reach convergence, which took only a few seconds on an ordinary laptop computer: about 10 seconds for Setting~1 and about 25 seconds for Setting~2.  

To summarize the performance, we consider five different measures.  First, we consider the mean inclusion probability for those variables in and out of the active set $S^\star$, respectively, i.e., 
\[ \bar p_1 = \frac{1}{s^\star} \sum_{j \in S^\star} p_j \quad \text{and} \quad \bar p_0 = \frac{1}{p-s^\star} \sum_{j \not\in S^\star} p_j. \]
We expect the former to be close to 1 and the latter to be close to 0.  Next, we consider the probability that the model selected by our empirical Bayes method, denoted by $\hat S$ is equal to or contains the true model $S^\star$.  Finally, we also compute the false discovery rate of our selection procedure.  A summary of these quantities for our empirical Bayes method, denoted by \emph{EB}, across the three settings is given in Tables~\ref{table:setting1}--\ref{table:setting3}.  

For comparison, we consider those methods discussed in \citet{narisetty.he.2014}, including their two Bayesian methods, denoted by BASAD and BASAD.BIC.  Two other Bayesian methods considered are the credible region approach of \citet{bondell.reich.2012}, denoted by BCR.Joint, and the spike-and-slab method of \citet{ishwaran.rao.2005.jasa, ishwaran.rao.2005.aos}, denoted by SpikeSlab.  We also consider three penalized likelihood methods, all tuned with BIC, namely, the lasso \citep{tibshirani1996}, the elastic net \citep{zou.hastie.2005}, and the smoothly clipped absolute deviation \citep{fanli2001}, denoted by Lasso.BIC, EN.BIC, and SCAD.BIC, respectively.  The results for these methods are taken from Tables~2--3 in \citet{narisetty.he.2014}, which were obtained based on 200 samples taken from the models described in Settings~1--3 described above.  

Our selection method based on our empirical Bayes posterior is the overall the best among those being compared in terms of selecting the true model and false discovery rate.  In addition to the strong finite-sample performance of our model selection procedure, our theory is arguably stronger than that available for the other methods in this comparison.  Take, for example, the BASAD method of \citet{narisetty.he.2014}, the next-best-performer in the simulation study.  Their method produces a posterior distribution for $\beta$ but since their prior has no point mass, this posterior cannot concentrate on a lower-dimensional subspace of $\RR^p$.  So, it is not clear if their posterior distribution for $\beta$ can attain the minimax concentration rate without tuning the prior using knowledge about the underlying sparsity level.

\begin{table}[t]
\begin{center}
\begin{tabular}{lccccc}
\hline
Method & $\bar p_0$ & $\bar p_1$ & $\prob(\hat S = S^\star)$ & $\prob(\hat S \supseteq S^\star)$ & FDR \\
\hline
BASAD & 0.001 & 0.948 & 0.730 & 0.775 & 0.011 \\
BASAD.BIC & 0.001 & 0.948 & 0.190 & 0.915 & 0.146 \\
BCR.Joint & & & 0.070 & 0.305 & 0.268 \\
SpikeSlab & & & 0.000 & 0.040 & 0.626 \\
Lasso.BIC & & & 0.005 & 0.845 & 0.466 \\
EN.BIC & & & 0.135 & 0.835 & 0.283 \\
SCAD.BIC & & & 0.045 & 0.980 & 0.328 \\
$EB$ & 0.002 & 0.959 & 0.680 & 0.795 & 0.051 \\
\hline
\end{tabular}
\end{center}
\caption{Simulation results for Setting~1.  First seven rows taken from Table~2 (top) in \citet{narisetty.he.2014}; the \emph{EB} row corresponds to our empirical Bayes procedure.}
\label{table:setting1}
\end{table}

\begin{table}[t]
\begin{center}
{
\begin{tabular}{lccccc}
\hline
Method & $\bar p_0$ & $\bar p_1$ & $\prob(\hat S = S^\star)$ & $\prob(\hat S \supseteq S^\star)$ & FDR \\
\hline
BASAD & 0.000 & 0.986 & 0.930 & 0.950 & 0.000 \\
BASAD.BIC & 0.000 & 0.986 & 0.720 & 0.990 & 0.046 \\
BCR.Joint & & & 0.090 & 0.250 & 0.176 \\
SpikeSlab & & & 0.000 & 0.050 & 0.574 \\
Lasso.BIC & & & 0.020 & 1.000 & 0.430 \\
EN.BIC & & & 0.325 & 1.000 & 0.177 \\
SCAD.BIC & & & 0.650 & 1.000 & 0.091 \\
$EB$ & 0.000 & 0.998 & 0.945 & 0.990 & 0.015 \\
\hline
\end{tabular}
}
\end{center}
\caption{Simulation results for Setting~2.  First seven rows taken from Table~2 (bottom) in \citet{narisetty.he.2014}; the \emph{EB} row corresponds to our empirical Bayes procedure.}
\label{table:setting2}
\end{table}

\begin{table}[t]
\begin{center}
\begin{tabular}{lccccc}
\hline
Method & $\bar p_0$ & $\bar p_1$ & $\prob(\hat S = S^\star)$ & $\prob(\hat S \supseteq S^\star)$ & FDR \\
\hline
BASAD & 0.002 & 0.622 & 0.185 & 0.195 & 0.066 \\
BASAD.BIC & 0.002 & 0.622 & 0.160 & 0.375 & 0.193 \\
BCR.Joint & & & 0.030 & 0.315 & 0.447 \\
SpikeSlab & & & 0.000 & 0.000 & 0.857 \\
Lasso.BIC & & & 0.000 & 0.520 & 0.561\\
EN.BIC & & & 0.040 & 0.345 & 0.478 \\
SCAD.BIC & & & 0.045 & 0.340 & 0.464 \\
$EB$ & 0.003 & 0.811 & 0.305 & 0.350 & 0.092 \\
\hline
\end{tabular}
\end{center}
\caption{Simulation results for Setting~3.  First seven rows taken from Table~3 in \citet{narisetty.he.2014}; the \emph{EB} row corresponds to our empirical Bayes procedure.}
\label{table:setting3}
\end{table}

\ifthenelse{1=1}{}{

NEW SIMULATIONS USING ALPHA = 0.999, GAMMA = 0.001

> ebreg.sim(reps=200, n=200, p=1000, r=0.25, beta=0.6 * 1:5, M=2000, a=0.05)
              pp0      pp1 pr.exact pr.contain        fdr     time
[1,] 0.0004016281 0.997559    0.945       0.99 0.01549405 22.99551
> 
> 
> ebreg.sim(reps=200, n=100, p=500, r=0.25, beta=0.6 * 1:5, M=2000, a=0.05)
            pp0      pp1 pr.exact pr.contain        fdr    time
[1,] 0.00206997 0.959215     0.68      0.795 0.05104239 8.28069
> 
> 
> ebreg.sim(reps=200, n=100, p=500, r=0.25, beta=0.6 * rep(1,5), M=2000, a=0.05)
            pp0      pp1 pr.exact pr.contain        fdr    time
[1,] 0.00261197 0.811447    0.305       0.35 0.09240267 8.19032

OLD SIMULATIONS...

CASE 1:

> o1 <- ebreg.sim(200, 100, 500, 0.25, 0.6 * 1:5, 2000, TRUE, 0.1); print(o1)
             pp0       pp1 pr.exact pr.contain        fdr
[1,] 0.003246894 0.9544405    0.635      0.755 0.06901045

> o1 <- ebreg.sim(200, 100, 500, 0.25, 0.6 * 1:5, 2000, TRUE, 0.05); print(o1)
             pp0     pp1 pr.exact pr.contain        fdr
[1,] 0.002204419 0.96558    0.745      0.835 0.04871404

CASE 2:

> o2 <- ebreg.sim(200, 200, 1000, 0.25, 0.6 * 1:5, 2000, TRUE, 0.1); print(o2)
             pp0       pp1 pr.exact pr.contain        fdr
[1,] 0.001356093 0.9954135    0.905      0.975 0.02745995

> o2 <- ebreg.sim(200, 200, 1000, 0.25, 0.6 * 1:5, 2000, TRUE, 0.05); print(o2)
             pp0       pp1 pr.exact pr.contain        fdr
[1,] 0.0002108719 0.9986405     0.95      0.995 0.01058442

CASE 3:

> o3 <- ebreg.sim(200, 100, 500, 0.25, 0.6 * rep(1, 5), 2000, TRUE, 0.1); o3
             pp0      pp1 pr.exact pr.contain       fdr
[1,] 0.004038485 0.776765    0.225       0.27 0.1082552

> o3 <- ebreg.sim(200, 100, 500, 0.25, 0.6 * rep(1, 5), 2000, TRUE, 0.05); o3
             pp0      pp1 pr.exact pr.contain       fdr
[1,] 0.004338581 0.786109     0.23       0.29 0.1009734

}


\section{Discussion}
\label{S:discuss}

We have presented an empirical Bayes model for the sparse high-dimensional regression problem.  Though the proposed approach has some unusual features, such as a data-dependent prior, we characterize the posterior concentration rate, which agrees with the optimal minimax rate in some cases.  To our knowledge, this is the only available minimax concentration rate result for a full posterior distribution in the sparse high-dimensional linear model.  Moreover, our formulation allows for relatively simple posterior computation, via Markov chain Monte Carlo, and simulation studies show that model selection by thresholding the posterior inclusion probabilities outperforms a variety of existing methods.   


The general strategy proposed here goes as follows.  Suppose we have a high-dimensional parameter, and different models $S$ identify a set of parameters $\theta_S$.  Suppose further that $\theta$ is sparse in the sense that only a few of its entries are non-null.  Then an empirical Bayes model is obtained by specifying a prior for $(S,\theta_S)$ as $\pi(S) \pi(d\theta_S \mid S)$, where $\pi(d\theta_S \mid S)$ would be allowed to depend on data through, say, the maximum likelihood estimator $\hat\theta_S$ of $\theta_S$.  Intuitively, the idea is to center the conditional prior on a data-dependent point, say $\hat\theta_S$, and then use the fractional likelihood to prevent the posterior to track the data too closely.  We believe this is a general tool that can be used in high-dimensional problems, and one possible application of this approach, which we plan to explore, is a mixture model where $S$ represents the number of mixture components, and $\theta_S$ is the set of parameters associated with a mixture model with $S$ mixture components.

\section*{Acknowledgments}

The authors are grateful for the valuable comments provided by the Editor, Associate Editor, and three anonymous referees.  This work is partially supported by the U.~S.~National Science Foundation, grants DMS--1506879 and DMS--1507073.

\appendix

\section{Proofs}
\label{S:proofs}

\ifthenelse{1=1}{}{
\subsection{Proof of Lemma~\ref{lem:denominator}}
\label{SS:proof.den}

Write $\|\cdot\|$ for the $\ell_2$-norm $\|\cdot\|_2$.  The denominator $D_n$ in \eqref{eq:post} involves an average over all suitable models $S$ with respect to $\pi(S)$.  This average is greater than the quantity for $S=S^\star$ times $\pi(S^\star)$.  That is, $D_n$ is bigger than 
\[ \pi(S^\star) \int e^{-\frac{\alpha}{2\sigma^2}\{\|Y-X_{S^\star}\beta_{S^\star}\|^2 - \|Y-X_{S^\star}\beta_{S^\star}^\star\|^2\}} \nm(\beta_{S^\star} \mid \hat\beta_{S^\star}, \gamma^{-1}(X_{S^\star}^\top X_{S^\star})^{-1}) \,d\beta_{S^\star}. \]
Direct calculation shows that the lower bound equals 
\[ \pi(S^\star) e^{\frac{\alpha}{2\sigma^2}\|X_{S^\star}(\hat\beta_{S^\star}-\beta_{S^\star}^\star)\|^2} \Bigr(1 + \frac{\alpha}{\gamma \sigma^2} \Bigr)^{-|S^\star|/2}. \]
According to the least-squares distribution theory, the quantity in the exponent is proportional to a $\chisq(|S^\star|)$ random variable so, with probability~1 for all sufficiently large $n$, the exponential term equals $e^{\frac{\alpha}{2\sigma^2}|S^\star|}$.  Then the lower bound can be written as 
\[ \pi(S^\star) e^{-\{\frac12\log(1 + \frac{\alpha}{\gamma \sigma^2}) - \frac12\frac{\alpha}{\sigma^2}\} \, |S^\star|}. \]
Under the stated conditions, the bracketed term in the exponent is positive.  Set $c$ to be this positive quantity to complete the proof. 
}

\subsection{Proof of Lemma~\ref{lem:numerator}}
\label{SS:proof.num}

Write $B_n=B_{\eps_n}$.  Rewrite the numerator $N_n$ of the posterior \eqref{eq:post} as 
\begin{align*}
N_n & = \int_{B_n} \sum_S \pi(S) \Bigl\{ \frac{\nm(Y \mid X\beta_{S+}, \sigma^2 I)}{\nm(Y \mid X \beta^\star, \sigma^2 I)} \Bigr\}^\alpha \nm(\beta_S \mid \hat\beta_S, \gamma^{-1}(X_S^\top X_S)^{-1}) \,d\beta_S \\
& = \sum_S \pi(S) \int_{B_n(S)} \Bigl\{ \frac{\nm(Y \mid X\beta_{S+}, \sigma^2 I)}{\nm(Y \mid X \beta^\star, \sigma^2 I)} \Bigr\}^\alpha \nm(\beta_S \mid \hat\beta_S, \gamma^{-1}(X_S^\top X_S)^{-1}) \,d\beta_S,
\end{align*}
where the sum is over all $S$ with $|S| \leq n$, $\beta_{S+}$ is a $p$-vector made by augmenting $\beta_S$ with $\beta_j = 0$ for all $j \in S^c$, and $B_n(S)$ is the set of all $\beta_S$ such that $\beta_{S+} \in B_n$.  Focus on a single $S$.  Taking expectation of the inner integral with respect to $Y \sim \nm(X\beta^\star, \sigma^2 I)$ gives 
\[ \int_{B_n(S)} \E\Bigl[ \Bigl\{ \frac{\nm(Y \mid X\beta_{S+}, \sigma^2 I)}{\nm(Y \mid X \beta^\star, \sigma^2 I)} \Bigr\}^\alpha \nm(\beta_S \mid \hat\beta_S, \gamma^{-1}(X_S^\top X_S)^{-1}) \Bigr] \,d\beta_S. \]
Apply H\"older's inequality to the inside expectation, i.e., for $h > 1$ and $q=(h-1)/h$, 
\begin{align}
\E\Bigl[ \Bigl\{ & \frac{\nm(Y \mid X\beta_{S+}, \sigma^2 I)}{\nm(Y \mid X \beta^\star, \sigma^2 I)} \Bigr\}^\alpha \nm(\beta_S \mid \hat\beta_S, \gamma^{-1}(X_S^\top X_S)^{-1}) \Bigr] \notag \\
& \leq \E^{1/h} \Bigl[ \Bigl\{\frac{\nm(Y \mid X\beta_{S+}, \sigma^2 I)}{\nm(Y \mid X \beta^\star, \sigma^2 I)} \Bigr\}^{h\alpha} \Bigl] \, \E^{1/q} \bigl[ \nm^q(\beta_S \mid \hat\beta_S, \gamma^{-1} (X_S^\top X_S)^{-1}) \bigr]. \label{eq:holder}
\end{align}
If $h\alpha < 1$, then a Renyi divergence formula is available for the first term, giving 
\begin{equation}
\label{eq:holder1}
\E^{1/h} \Bigl[ \Bigl\{\frac{\nm(Y \mid X\beta_{S+}, \sigma^2 I)}{\nm(Y \mid X \beta^\star, \sigma^2 I)} \Bigr\}^{h\alpha} \Bigl] = e^{-\frac{\alpha (1-h\alpha)}{2\sigma^2} \|X(\beta_{S+} - \beta^\star)\|^2}. 
\end{equation}
For the second term in the product above, recall that $\hat\beta_S = (X_S^\top X_S)^{-1} X_S^\top Y$.  Then 
\[ X_S \beta_S - X_S \hat\beta_S = X_S (X_S^\top X_S)^{-1} X_S^\top (X_S \beta_S - Y), \]
and, therefore, since $X_S (X_S^\top X_S)^{-1} X_S^\top$ is idempotent of rank $|S|$, we get that
\[ Z := \tfrac{1}{\sigma^2}\|X_S \beta_S - X_S \hat\beta_S\|^2 = \tfrac{1}{\sigma^2}\|X_S (X_S^\top X_S)^{-1} X_S^\top (X_S \beta_S - Y)\|^2 \]
is distributed as a non-central chi-square with $|S|$ degrees of freedom and non-centrality parameter $\lambda = \frac{1}{\sigma^2}\|X_S(\beta_S - (X_S^\top X_S)^{-1}X_S^\top X \beta^\star)\|^2$.  Then 
\begin{align}
\E^{1/q} \bigl[ &\nm^q(\beta_S \mid \hat\beta_S, \gamma^{-1} (X_S^\top X_S)^{-1}) \bigr] \notag \\
& = \frac{\gamma^{|S|/2}|X_S^\top X_S|^{1/2}}{(2\pi)^{|S|/2}} \E^{1/q} (e^{-\frac{q\gamma}{2} Z}) \notag \\
& = \frac{\gamma^{|S|/2}|X_S^\top X_S|^{1/2}}{(2\pi)^{|S|/2}} (1 + q\gamma)^{-\frac{|S|}{2q}} e^{-\frac{\gamma}{2(1+q\gamma)} \lambda} \notag \\
& = \frac{\gamma^{|S|/2}|X_S^\top X_S|^{1/2}}{(2\pi)^{|S|/2}} (1 + q\gamma)^{-\frac{|S|}{2q}} e^{-\frac{\gamma}{2\sigma^2(1+q\gamma)} \|X_S(\beta_S - (X_S^\top X_S)^{-1}X_S^\top X \beta^\star)\|^2}, \label{eq:holder2} 
\end{align}
where the second equality is from the standard formula for the moment generating function of a non-central chi-square random variable.  Now we must integrate the upper bound \eqref{eq:holder} over $A_n(S)$ with respect to $\beta_S$.  It is clear from the definition of $B_n(S)$ that the quantity in \eqref{eq:holder1} is bounded on $B_n(S)$, i.e., 
\[ e^{-\frac{\alpha (1-h\alpha)}{2\sigma^2} \|X(\beta_{S+} - \beta^\star)\|^2} \leq e^{-\frac{\alpha (1-h\alpha)}{2\sigma^2} \eps_n}, \quad \beta_S \in B_n(S). \]
It is also clear that the expression \eqref{eq:holder2} resembles a normal density in $\beta_S$, modulo some multiplicative factors.  The algebra is tedious, but the integral of \eqref{eq:holder2} with respect to $\beta_S$ is bounded above by 
\[ \phi^{|S|} \quad \text{where} \quad \phi = \Bigl\{ \frac{(1 + q\gamma \sigma^2)^{1-\frac1q}}{\sigma^2} \Bigr\}^{\frac12}. \]
Putting everything together, we have that 
\[ \E(N_n) \leq e^{-\frac{\alpha (1-h\alpha)}{2\sigma^2} \eps_n} \sum_S \phi^{|S|} \pi(S). \]
Taking $d = \alpha (1-h\alpha) / 2\sigma^2$ completes the proof.

\subsection{Proof of Lemma~\ref{lem:dim}}
\label{SS:proof.dim}

The proof is an application of ideas used in the proof of Lemma~\ref{lem:numerator}.  In particular, $N_n(\Delta)$ equals 
\[ \sum_{S: \Delta \leq |S| \leq n} \pi(S) \int \Bigl\{ \frac{\nm(Y \mid X\beta_{S+}, \sigma^2 I)}{\nm(Y \mid X \beta^\star, \sigma^2 I)} \Bigr\}^\alpha \nm(\beta_S \mid \hat\beta_S, \gamma^{-1}(X_S^\top X_S)^{-1}) \,d\beta_S, \]
Take expectation with respect to $Y \sim \nm(X\beta^\star, \sigma^2 I)$ as in the proof of Lemma~\ref{lem:denominator} and move expectation to the inside of the integral.  Working with each $S$ term separately, apply H\"older's inequality to bound the expectation of the product.  This upper bound consists of a product of three terms just like in the previous proof.  The first is bounded by 1; the second is $\phi^{|S|}$; and the third is a probability density function in $\beta_S$.  Then the integral over $\beta_S$ is bounded by $\phi^{|S|}$ and the claim follows.

\subsection{Proof of Theorem~\ref{thm:selection1}}
\label{proof:selection1}

Fix $\beta^\star$ and write $S^\star = S_{\beta^\star}$ as usual.  Write $P_S$ for the $n \times n$ matrix projecting onto the column space of $X_S$.  If $S \supset S^\star$, then 
\[ \|Y - \hat Y_{S^\star}\|^2 - \|Y - \hat Y_S\|^2 = Y^\top (P_S - P_{S^\star}) Y, \]
and, since $P_S - P_{S^\star}$ is idempotent of rank $|S|-|S^\star|$, this quantity is distributed as a non-central chi-square with $|S|-|S^\star|$ degrees of freedom and non-centrality parameter
\[ (X\beta^\star)^\top (P_S - P_{S^\star}) (X \beta^\star). \]
By definition of $P_S$, it turns out that the non-centrality parameter in the above display is zero, so it is actually an ordinary/central chi-square.  From the chi-square moment generating function we immediately get 
\[ \E_{\beta^\star}\{\pi^n(S)\} \leq \frac{\pi(S)}{\pi(S^\star)} z^{|S^\star|-|S|}, \]
where $z$ is a constant that depends only on $(\alpha,\gamma, \sigma^2)$.  Then
\[ \E_{\beta^\star}\{\Pi^n(\beta: S_\beta \supset S^\star)\} = \sum_{S: S \supset S^\star} \E_{\beta^\star}\{\pi^n(S)\} \leq \sum_{S: S \supset S^\star} \frac{\pi(S)}{\pi(S^\star)} z^{|S^\star|-|S|}. \]
Plug in our complexity prior and simplify the upper bound:
\[ \sum_{s > s^\star} \frac{\binom{p-s^\star}{p-s} \binom{p}{s^\star}}{\binom{p}{s}} \Bigl( \frac{z}{cp^a} \Bigr)^{s-s^\star}. \]
From 
\[ \frac{\binom{p-s^\star}{p-s} \binom{p}{s^\star}}{\binom{p}{s}} = \binom{s}{s^\star} = \binom{s}{s-s^\star} \leq s^{s-s^\star}, \]
the upper bound becomes 
\[ \sum_{s = s^\star}^R \Bigl( \frac{z s}{cp^a} \Bigr)^{s-s^\star} \leq \frac{z R}{cp^a} \times O(1). \]
So, if $a$ is such that $p^a \gg R$, the upper bound vanishes, completing the proof.

\subsection{Proof of Theorem~\ref{thm:selection2}}
\label{proof:selection2}

In light of Theorem~\ref{thm:selection1}, it suffices to show that $\E_{\beta^\star}\{\Pi^n(\beta: S_\beta \subset S^\star)\} \to 0$.  To start, take a generic $S \subset S^\star$.  Then, from \eqref{eq:model.bound}, we have 
\[ \frac{\pi^n(S)}{\pi^n(S^\star)} = \frac{\pi(S)}{\pi(S^\star)} \nu^{s^\star-|S|} e^{-\frac{\alpha}{2\sigma^2}\{\|Y-\hat Y_S\|^2 - \|Y-\hat Y_{S^\star}\|^2\}}. \]
The exponent $Z_S := \frac{1}{\sigma^2}\{\|Y-\hat Y_S\|^2 - \|Y-\hat Y_{S^\star}\|^2\}$ is a chi-square random variable with $s^\star-|S|$ degrees of freedom and non-centrality parameter 
\[ \lambda_S := \tfrac{1}{\sigma^2} (X\beta^\star)^\top (P_{S^\star}-P_S) (X \beta^\star). \]
The algebra is a bit tedious, but we can simplify $\lambda_S$ as 
\[ \lambda_S = \tfrac{1}{\sigma^2} \|(I-P_S) X_{S^\star \cap S^c} \beta_{S^\star \cap S^c}^\star\|^2. \]
From the non-central chi-square moment generating function we have
\[ \E_{\beta^\star}\{ \pi^n(S) \} \leq \frac{\pi(S)}{\pi(S^\star)} z^{s^\star-|S|} e^{-\frac{\alpha}{2(1+\alpha)}\lambda_S}, \]
where $z=\nu / (1 + \alpha)$.  The irrepresentability result in Lemma~5 of \citet{ariascastro.lounici.2014} gives a lower bound on $\lambda_S$:
\[ \lambda_S \geq \tfrac{1}{\sigma^2} \kappa(|S^\star \cap S^c|)^2 \|\beta_{S^\star \cap S^c}^\star \|^2. \]
Monotonicity of $\kappa$ implies that 
\[ \kappa(|S^\star \cap S^c|) \geq \kappa(s^\star) > 0 \]
and, furthermore, by the beta-min condition,  
\[ \|\beta_{S^\star \cap S^c}^\star\|^2 \geq \rho_n^2(s^\star-|S|). \]
Putting everything together, including the definition of $\rho_n$, we get 
\begin{align*}
\E_{\beta^\star}\{ \pi^n(S) \} & \leq \frac{\pi(S)}{\pi(S^\star)} z^{s^\star-|S|} e^{-\frac{\alpha}{2(1+\alpha)}\frac{1}{\sigma^2}\kappa(s^\star)^2 \rho_n^2(s^\star-|S|)} \\
& = \frac{\pi(S)}{\pi(S^\star)} (z p^{-M})^{s^\star - |S|}.
\end{align*}
If we can show that the sum of our upper bound above, over all $S \subset S^\star$, vanishes, then we are done.   Plugging in our complexity prior, we need to bound 
\[ \sum_{s < s^\star} \frac{\binom{s^\star}{s} \binom{p}{s^\star}}{\binom{p}{s}} (z cp^{a-M})^{s^\star-s}. \]
where $r$ is a constant that depends only on $(\alpha, \gamma, \sigma^2)$.  Note that 
\[ \frac{\binom{s^\star}{s} \binom{p}{s^\star}}{\binom{p}{s}} = \binom{p-s}{p-s^\star} = \binom{p-s}{s^\star-s} \leq p^{s^\star-s}. \]
Then the summation can be bounded above by
\[ \sum_{s < s^\star} (z cp^{a+1-M})^{s^\star - s} \leq p^{a+1-M} \times O(1), \]
where the inequality follows from the formula for partial sums of a geometric series.  Since $M > a+1$, the upper bound vanishes, completing the proof.

\section{Justification of claims in Remark~\ref{re:minimax}}
\label{S:claims}


Consider a prior $\tilde \pi$ for the model $S$ of the form 
\[ \tilde \pi(S) = (1 - w_n) \pi(S) + w_n \delta_{S_0}(S), \]
where $\pi$ is a prior on models of size $|S| \leq R$, for $R=\min\{n, \text{rank}(X)\}$, $S_0$ is a fixed model with $|S_0|=R$ and $\text{span}(X_{S_0}) = \text{span}(X)$, and $w_n = e^{-rR}$ for $r > 0$ to be determined; a similar setup is taken in \citet[][Sec.~5.10]{gao.vdv.zhou.2015}, and the choice $w_n \equiv 1/2$ is considered in \citet{rigollet.tsybakov.2011}, Equation~(3.4).  With the prior $\tilde \pi$, it is easy to see that the denominator $D_n$ of the posterior satisfies
\[ D_n \geq \tfrac12 \max\bigl\{ \pi(S^\star) g(S^\star), w_n g(S_0) \bigr\} \]
for sufficiently large $n$ (so that $1-w_n > \frac12$, say), where 
\[ g(S) = e^{\frac{\alpha}{2\sigma^2} \{\|X_{S^\star}(\hat\beta_{S^\star} - \beta_{S^\star}^\star)\|^2 + \|Y-X_{S^\star}\hat\beta_{S^\star}\|^2 - \|Y-X_S \hat\beta_S\|^2\}} \Bigl( 1 + \frac{\alpha}{\gamma \sigma^2} \Bigr)^{-|S|/2}. \]
For the case $S=S^\star$, in the first term of the maximum, we have 
\[ g(S^\star) = e^{\frac{\alpha}{2\sigma^2} \|X_{S^\star}(\hat\beta_{S^\star} - \beta_{S^\star}^\star)\|^2} \Bigl( 1 + \frac{\alpha}{\gamma \sigma^2} \Bigr)^{-|S^\star|/2}, \]
just like in the proof of Lemma~\ref{lem:denominator}.  Hence, 
\[ g(S^\star) \geq \Bigl( 1 + \frac{\alpha}{\gamma \sigma^2} \Bigr)^{-|S^\star|/2}. \]
For the second term, since $|S| \leq R$, we have 
\[ g(S_0) \geq e^{\frac{\alpha}{2\sigma^2} \{\|X_{S^\star}(\hat\beta_{S^\star} - \beta_{S^\star}^\star)\|^2 + \|Y-X_{S^\star}\hat\beta_{S^\star}\|^2 - \|Y-X_{S_0} \hat\beta_{S_0}\|^2\}} \Bigl( 1 + \frac{\alpha}{\gamma \sigma^2} \Bigr)^{-R/2}. \]
Since the span of $X_{S_0}$ contains that of $X_{S^\star}$, by assumption, we have that  
\[ \|Y-X_{S^\star}\hat\beta_{S^\star}\|^2 \geq \|Y-X_{S_0} \hat\beta_{S_0}\|^2 \]
and, consequently, the term in the exponent above is bigger than $\frac{\alpha}{2\sigma^2}\|X_{S^\star}(\hat\beta_{S^\star} - \beta_{S^\star}^\star)\|^2 $, which is obviously positive.  Therefore, the second term in the maximum is 
\[ \geq w_n e^{-cR} e^{\frac{\alpha2} \|X_{S^\star}(\hat\beta_{S^\star} - \beta_{S^\star}^\star)\|^2} > w_n e^{-c R} = e^{-}{2\sigma^(r+c) R}, \]
where $c=\frac12 \log(1 + \frac{\alpha}{\gamma\sigma^2})$ is as in Lemma~\ref{lem:denominator}.  Finally, for $A=r+c$, we have 
\[ D_n \geq \tfrac12 \max\{\pi(S^\star)e^{-c|S^\star|}, e^{-A R}\} \]
with probability~1, for large $n$, as desired.  

We claim that, with this new prior $\tilde\pi$, the posterior can achieve the minimax rate for the prediction loss under both the ordinary and ultra high-dimensional regimes.  That is, we get the optimal rate 
\[ \eps_n = \min\bigl\{R, s^\star \log(p / s^\star)\}. \]
This follows easily from the denominator bound discussed above, so long as our current numerator bound from Lemma~\ref{lem:numerator} also holds for the new prior.  The majority of the proof of Lemma~\ref{lem:numerator} has nothing to do with the model prior, so we can immediately jump to the following conclusion:
\[ \E_{\beta^\star}(N_n) \leq e^{-k \eps_n} \sum_{S: |S| \leq R} \phi^{|S|} \tilde \pi(S), \]
where $d$ and $\phi$ are as in the proof of Lemma~\ref{lem:numerator}. Now, for the weighted average part, we have 
\[ \sum_S \phi^{|S|} \tilde \pi(S) \leq \sum_S \phi^{|S|} \pi(S) + w_n \phi^R. \]
The first term in this upper bound is just like that in the proof of Lemma~\ref{lem:numerator}, so we have a handle on this.  We need to choose $w_n$ in such a way that the second term is also controlled.  Since $w_n = e^{-r R}$ for some $r > 0$, it follows that we need $r \geq \log \phi$.  With this choice, the optimal minimax rate can be achieved in both ordinary and ultra high-dimensional regimes.  

We claimed in Remark~\ref{re:minimax} that there is a price to be paid, in terms of model selection performance, if one uses the prior $\tilde\pi$ discussed above.  The problem is that the weight $w_n$ assigned to the large model $S_0$, with $|S_0|=R > s^\star$, is considerably larger than the weight $(1-w_n)\pi(S^\star)$ assigned to the true model $S^\star$.  Then the corresponding posterior mass assigned to $S_0$ is too large, large enough to pull the posterior away from the true model, leading to inconsistency.

\ifthenelse{1=0}{
\bibliographystyle{apalike}
\bibliography{/Users/rgmartin/Dropbox/Research/mybib}
}
{
}

\pagebreak

\subsection*{Corrections}

Since the paper has been published, a few relatively minor missteps have been identified.  These do not significantly affect the results presented in the paper, but they are worth pointing out here for the sake of readers trying to follow along with the arguments.  We have chosen to keep the original (published) version of the paper as is, but just add a section here to record and correct these missteps.

\subsubsection*{C1. \; Posterior for the model}

The formula for $\pi^n(S)$, the marginal posterior for the model $S$, given in Section~4.1 of the paper has a small typo.  In particular, the formula should read 
\[ \pi^n(S) \propto \pi(S) e^{-\frac{\alpha}{2\sigma^2} \|Y-\hat Y_S\|^2} \Bigl( 1 + \frac{\alpha}{\gamma \sigma^2} \Bigr)^{-|S|/2}. \]
After making this correction, it now becomes clear that the choice ``$\gamma=0$'' used in some of the numerical illustrations is not a feasible one, although that extreme choice was only shown to demonstrate that there is not a singularity in performance at the boundary.  This typo also made it into the R codes used to compute the method, but it has since been confirmed that the numerical results presented in the paper (for small but non-zero $\gamma$) hold virtually unchanged with the correct formula for $\pi^n(S)$ above.  The corrected R codes are now available at \url{www4.stat.ncsu.edu/~rmartin}.  The authors thanks two PhD students---Mr.~Chang Liu and Ms.~Yue Yang---at NC State for spotting this mistake and redoing the simulations with the correct posterior.

\subsubsection*{C2. \; Strengthening the result in Theorem~\ref{thm:selection1}}

Theorem~\ref{thm:selection1} requires that $R/p^a \to 0$, where $a$ is a parameter involved in the complexity prior.  This can be a limitation since it may only hold when $p$ is much larger than $n$.  However, in the proof of Theorem~\ref{thm:selection1}, one can employ the dimensionality result in Theorem~\ref{thm:dim} to reduce the size of configurations, $S$, under consideration.  As a result, Theorem~\ref{thm:selection1} holds under the considerably weaker assumption that $s^\star / p^a \to 0$.

\subsubsection*{C3. \; Model selection consistency}

Theorem~4 in the paper shows that the posterior mass assigned to supersets of the true model $S^\star$ is vanishing, and the proof of Theorem~5 shows the same result for subsets of $S^\star$.  But model selection consistency requires that the mass assigned to {\em all models not equal to $S^\star$} must be vanishing.  This model selection consistency does hold, but it requires some inconsequential changes to the conditions of Theorem~5.  In particular, note that the lower bound in the new beta-min condition below only differs from that in Equation~(18) in the paper by constants, not by rate.  The authors thank Dr.~Kyoungjae Lee at the University of Notre Dame for pointing out the gap in our original proof.  

\begin{thm0}
Take any $\beta^\star$ with $S^\star=S_{\beta^\star}$ such that $|S^\star|=s^\star$ and 
\[ \min_{j \in S^\star} |\beta_j^\star| \geq \rho_n := \frac{\sigma}{\kappa(C's^\star)} \Bigl\{ \frac{2M}{\alpha(1-\alpha)} \log p \Bigr\}^{1/2}, \]
where $C'=1+C$ for $C$ as in Theorem~2 and $M$ is a constant with $M > 1 + a$, with $a > 0$ the power in the complexity prior.  This implicitly assumes that $\kappa(C'|S^\star|) > 0$.  Then 
\[ \E_{\beta^\star}\{\Pi^n(\beta: S_\beta = S^\star)\} \to 1. \]
\end{thm0}

\begin{proof}
In light of Theorem~4, it suffices to show that $\E_{\beta^\star}\{\Pi^n(\beta: S_\beta \not\supseteq S^\star)\} \to 0$.  To start, take a generic $S \not\supseteq S^\star$; by Theorem~2, it suffices to consider only those $S$ for which $|S| \leq C|S^\star|$.  Then we have  
\[ \|Y-\hat Y_{S^\star}\|^2 - \|Y - \hat Y_S\|^2 = Y^\top(P_S - P_{S^\star}) Y, \]
and if we plug in $Y=X\beta^\star + \sigma \eps$, where $\eps \sim \nm_n(0,I)$, then we get 
\[ -\|(I-P_S)X\beta^\star\|^2 - 2 \sigma \eps^\top (I-P_S)X\beta^\star + \sigma^2 \eps^\top (P_S - P_{S^\star}) \eps. \]
Bound the right-most quadratic form above as follows 
\begin{align*}
\eps^\top (P_S - P_{S^\star}) \eps & = \eps^\top (P_S - P_{S \cap S^\star}) \eps - \eps^\top (P_{S^\star}-P_{S \cap S^\star}) \eps \\
& \leq \eps^\top (P_S - P_{S \cap S^\star}) \eps,
\end{align*}
which follows since $P_{S^\star} - P_{S \cap S^\star}$ is positive definite.  This implies that, up to a multiplicative constant, the exponent on the right-hand side of Equation~(17) in the paper is bounded above by
\[ -\|(I-P_S)X\beta^\star\|^2 - 2 \sigma \eps^\top (I-P_S)X\beta^\star + \sigma^2 \eps^\top (P_S - P_{S \cap S^\star}) \eps. \]
The key observation now is that $(I-P_S)(P_S - P_{S \cap S^\star})=0$, which implies that 
\[ \eps^\top (I-P_S)X\beta^\star \quad \perp \quad \eps^\top (P_S - P_{S \cap S^\star}) \eps. \]
We also know the marginal distributions are normal and chi-square, respectively, and using their moment generating functions, we get  
\[ \E_{\beta^\star}\bigl[ e^{\frac{\alpha}{2\sigma^2}\{\|Y-\hat Y_{S^\star}\|^2 - \|Y-\hat Y_S\|^2\}} \bigr] \leq (1-\alpha)^{-\frac12 (|S|-|S \cap S^\star|)} e^{-\frac{\alpha(1-\alpha)}{2\sigma^2} \|(I-P_S)X\beta^\star\|^2}. \]
Since 
\[ \|(I-P_S) X\beta^\star\|^2 = \|(I-P_S) X_{S^\star \cap S^c} \beta_{S^\star \cap S^c}^\star\|^2, \]
it follows from Lemma~5 of Arias-Castro and Lounici (2014) that 
\[ \|(I-P_S) X\beta^\star\|^2 \geq \kappa(|S^\star \cup S|)^2 (|S^\star|-|S \cap S^\star|) \rho_n^2. \]
Next, for $C'$ as above, we have $|S \cup S^\star| \leq |S| + |S^\star| \leq C'|S^\star|$ which, together with monotonicity of $\kappa$ and the definition of $\rho_n$ in the beta-min condition, gives 
\[ \tfrac{\alpha(1-\alpha)}{2\sigma^2} \|(I-P_S)X\beta^\star\|^2 \geq M (|S^\star|-|S \cap S^\star|) \log p. \]
Therefore, 
\[ \E_{\beta^\star}\{\pi^n(S)\} \leq \frac{\pi(S)}{\pi(S^\star)} v^{|S^\star|-|S|} (wp^{-M})^{|S^\star|-|S \cap S^\star|}, \]
where $w=(1-\alpha)^{-1/2}$.  Plug in the complexity prior and then sum over all $S$ such that $S \not\supseteq S^\star$ (and of size no larger than $Cs^\star$, by Theorem~2) to get 
\[ \E_{\beta^\star} \{ \Pi^n(\beta: S_\beta \not\supseteq S^\star) \} \leq \sum_{s=0}^{Cs^\star} \sum_{t=1}^{\min(s,s^\star)} \frac{\binom{s^\star}{t} \binom{p-s^\star}{s-t} \binom{p}{s^\star}}{\binom{p}{s}} (vcp^a)^{s^\star-s} (wp^{-M})^{s^\star-t}. \]
Here $t$ represents the number of variables in both $S$ and $S^\star$ and, since $S \not\supseteq S^\star$, $t$ must be less than $s^\star$.  The ratio of binomial coefficients simplifies as 
\[ \frac{\binom{s^\star}{t} \binom{p-s^\star}{s-t} \binom{p}{s^\star}}{\binom{p}{s}} = \binom{s}{t} \binom{p-s}{s^\star - t} \leq s^{s-t} p^{s^\star - t}. \] 
Plug in this bound and split the sum over $s$ into two cases: $s \leq s^\star - 1$ and $s \geq s^\star$.  For the first case, since $M > 1 + a$, we have 
\[ \sum_{s=0}^{s^\star-1} \sum_{t=0}^s (vcp^a/s)^{s^\star-s} (wsp^{1-M})^{s^\star - t} \lesssim \sum_{s=0}^{s^\star-1} (vcp^{1+a-M})^{s^\star-s} \to 0. \]
Similarly, for the second case
\[ \sum_{s=s^\star}^{Cs^\star} \sum_{t=0}^{s^\star-1} (vcp^a/s)^{s^\star-s} (wsp^{1-M})^{s^\star-t} \lesssim s^\star p^{1-M} \sum_{s=s^\star}^{Cs^\star} (vcp^a/s)^{s^\star-s}, \]
and, since $p^a \gg R \gg s^\star$, the sum is dominated by $p^{1-M} \to 0$.  In either case, the upper bound vanishes which implies that $\E_{\beta^\star} \{ \Pi^n(\beta: S_\beta \not\supseteq S^\star) \} \to 0$ as was to be shown.  
\end{proof}

\end{document}